\documentclass[reqno,11pt]{amsart}

\usepackage{mathdots}
\usepackage{tikz}
\usepackage{tikz-cd}
\usepackage{amssymb}
\usepackage{amsgen}
\usepackage{amsmath}
\usepackage{amsthm}
\usepackage{mathrsfs}
\usepackage{cite}
\usepackage{amsfonts}
\usepackage{enumitem}

\newcommand{\image}{\operatorname{im}}
\newcommand{\coker}{\operatorname{coker}}
\newcommand{\ab}[1]{#1_{\mathrm {ab}}}

\newcommand{\rk}{\operatorname{\mathrm{rk}}\nolimits}

\newcommand{\inv}{^{-1}}
\newcommand{\p}{\varphi}

\newcommand{\ov}[1]{\ensuremath{\overline {#1}}}
\newcommand{\til}[1]{\ensuremath{\widetilde {#1}}}

\newcommand{\Hom}{\operatorname{\mathrm{Hom}}\nolimits}
\newcommand{\Tor}{\operatorname{\mathrm{Tor}}\nolimits}
\newcommand{\Ext}{\operatorname{\mathrm{Ext}}\nolimits}

\usepackage{xcolor}

\newtheorem*{ThmA*}{Theorem~A}
\newtheorem{Thm}{Theorem}[section]
\newtheorem{Prop}[Thm]{Proposition}

\newtheorem{Lemma}[Thm]{Lemma}
{\theoremstyle{definition}
}
{\theoremstyle{remark}
}
\newtheorem{Cor}[Thm]{Corollary}
{\theoremstyle{remark}
}
{\theoremstyle{remark}
}

{\theoremstyle{remark}
}
{\theoremstyle{remark}
}
{\theoremstyle{remark}
}
{\theoremstyle{remark}
\newtheorem*{Claim*}{Claim}}

\numberwithin{equation}{section}

\title{The homology of completely simple semigroups}

\author{Benjamin Steinberg}
\address[B.~Steinberg]{%
    Department of Mathematics\\
    City College of New York\\
    Convent Avenue at 138th Street\\
    New York, New York 10031\\
    USA}
\email{bsteinberg@ccny.cuny.edu}

\thanks{The author was supported by a Simons Foundation Collaboration Grant, award number 849561, and the Australian Research Council Grant DP230103184.}
\date{May 10, 2024}


\begin{document}

\begin{abstract}
I explicitly compute the Eilenberg-Mac Lane homology of a completely simple semigroup using topological means.  I also complete Gray and Pride's investigation into the homological finiteness properties of completely simple semigroups, as well as studying their topological finiteness properties.  I give a topological proof of Pride's unpublished homological lower bound for the deficiency of a monoid or semigroup. 
\end{abstract}

\maketitle
\section{Introduction}

My goal here is to use a simple topological argument to compute the homology of a completely simple semigroup.  I'll also sketch how one might obtain the same result using algebraic means, although it requires some work to verify that a particular boundary map from a long exact homology sequence  behaves as the topological approach says it ought to do.  The topological approach also allows me establish analogues of the results of Gray and Pride~\cite{GrayPride} on homological finiteness properties of completely simple semigroups for the topological finiteness properties introduced in~\cite{TopFinite1}. Moreover, I complete Gray and Pride's study~\cite{GrayPride} of their homological finiteness properties by handling the case of where the semigroup contains infinitely many minimal left ideals. I also consider the cohomology of completely simple semigroups, obtaining more precise results than previous work of Nico~\cite{Nicocohom1}, but not a complete computation.

 Note that the second homology of a finite simple semigroup was computed using Squier's resolution associated to a complete rewriting system in~\cite{effsimple} in connection with the study of efficiency.  Pride, in an unpublished result, showed that the usual lower bound on deficiency for group presentations also works for monoids (and semigroups).  I'll give here a topological proof of Pride's result  so that it can be found in the literature.

Recall that a semigroup is \emph{completely simple} if it has no proper two-sided ideals and it contains a primitive idempotent (or, equivalently, both a minimal left and right ideal).  Rees's theorem~\cite{Rees} says that any completely simple semigroup is isomorphic to a Rees matrix semigroup $\mathcal M(G,A,B,C)$ over a group $G$.  Here $A,B$ are sets, $C\colon B\times A\to G$ is a matrix and $\mathcal M(G,A,B,C)=A\times G\times B$ with multiplication \[(a,g,b)(a',g',b') = (a,gC_{ba'}g',b').\]  It is well known~\cite{CP} that if you fix $a_0\in A$ and $b_0\in B$, then you may assume without loss of generality that $C_{b_0a}=1=C_{ba_0}$ for all $a\in A$ and $b\in B$.  In this case $C$ is said to be \emph{normalized} with respect to $a_0,b_0$.   Note that any finite simple semigroup is completely simple.

Let me now recall the definition of monoid and semigroup homology.  If $M$ is a monoid, then $\mathbb ZM$ denotes the monoid ring of $M$ and $\mathbb Z$ denotes (abusively), both the trivial left and right $\mathbb ZM$-modules.  Then, by definition, $H_n(M)=\Tor^{\mathbb ZM}_n(\mathbb Z,\mathbb Z)$.  This can be described more concretely as follows.  If $A$ is a left $\mathbb ZM$-module, then $A_M = A/\langle ma-a\mid m\in M,a\in A\rangle$.  Note that $A_M\cong \mathbb Z\otimes_{\mathbb ZM} A$.  In particular, if $P$ is a (finitely generated) projective $\mathbb ZM$-module, then $P_M$ is a (finitely generated) free abelian group.  For any projective resolution $P_{\bullet}\twoheadrightarrow \mathbb Z$ one has that $H_n(M)= H_n((P_{\bullet})_M)$.  If $S$ is a semigroup, define $H_n(S)=H_n(S^1)$ where $S^1$ is obtained by adjoining an identity to $S$.  If $M$ is a monoid, one can check that $H_n(M)\cong H_n(M^1)$ since the normalized bar resolution of $\mathbb Z$ for $M^1$ is precisely the bar resolution of $\mathbb Z$ for $M$.  Therefore, the homology of a monoid is the same viewed as either a semigroup or a monoid.  Note that there are other, more sophisticated, homology theories for monoids~\cite{Leechcohom},  but the one of interest here is the classical theory.  The cohomology of a monoid $M$ is $H^n(M)=\Ext^n_{\mathbb ZM}(\mathbb Z,\mathbb Z)$; the cohomology of a semigroup $S$ is $H^n(S^1)$.

The abelianization $G/[G,G]$ of a group $G$ is denoted $\ab G$. The main result of this paper is the following.

\begin{ThmA*}
Let $S=\mathcal M(G,A,B,C)$ be a completely simple semigroup over a group $G$.  Assume that $C$ is normalized with respect to $a_0\in A$ and $b_0\in B$, and put $\ov A=A\setminus\{a_0\}$, $\ov B=B\setminus \{b_0\}$.  Let $\psi\colon \mathbb Z[\ov A\times \ov B]\to \ab G$ be the homomorphism with $\psi(a,b) =C_{ba}[G,G]$.  Then
\[H_n(S) = \begin{cases}\mathbb Z, & \text{if}\ n=0\\ \coker \psi, & \text{if}\ n=1\\ H_2(G)\oplus \ker \psi, & \text{if}\ n=2\\ H_n(G), & \text{if}\ n\geq 3.\end{cases}\]  In particular, if $\left\langle C_{ba}[G,G]\mid (a,b)\in \ov A\times \ov B\right\rangle$ is finite (i.e., if $G$ is finite or if $C_{ba}=1$ for all $(a,b)\in \ov A\times \ov B$), then
\[H_n(S) = \begin{cases}\mathbb Z, & \text{if}\ n=0\\ \coker \psi, & \text{if}\ n=1\\ H_2(G)\oplus \mathbb Z[\ov A\times \ov B], & \text{if}\ n=2\\ H_n(G), & \text{if}\ n\geq 3\end{cases}\] since any finite index subgroup of $\mathbb Z[\ov A\times \ov B]$ is free of the same rank.
\end{ThmA*}

The proof of Theorem~A proceeds along the following lines.  Starting with an Eilenberg-Mac Lane space $Y$ for $G$, I'll show that the homology of $S$ can be computed as the homology of the space $X$ obtained from $Y$ by gluing in, for each $(a,b)\in \ov A\times\ov B$, a disk attached to a loop representing $C_{ba}$ in $\pi_1(Y)\cong G$.  A simple application of the long exact sequence in relative homology for the CW pair $(X,Y)$ (or  the Mayer-Vietoris sequence) then easily produces Theorem~A.  

\subsection*{Acknowledgments}
I would like to thank Robert D.~Gray for some stimulating email exchanges that prompted me to investigate the exact conditions under which homological finiteness properties hold in complete generality.

\section{Equivariant classifying spaces and Pride's theorem}
The notions in this section were introduced in~\cite{TopFinite1} for monoids.  The analogous theory for groups is classical, cf.~\cite{Browncohomology}.
Fix a monoid $M$.   I shall denote by $B^n$ the topological $n$-ball and $S^{n-1}$ its boundary sphere.  Note that $B^0$ is a point and $S^{-1}=\emptyset$.  If $A$ is a left $M$-set and $X$ is a topological space, then $A\times X$ is an $M$-space with action $m(a,x) = (ma,x)$ for $a\in A$ and $x\in X$.  So $A\times X$ is a disjoint union of $|A|$ copies of $X$, which are rigidly shuffled about by  $M$ according to the action of $M$ on $A$.

 A \emph{projective $M$-cell} of dimension $n$ is an $M$-space of the form $Me\times B^n$ with $e\in M$ idempotent.  If $e=1$, I call it a \emph{free $M$-cell}.  Any projective $G$-cell for a group $G$ is a free $G$-cell.    A \emph{projective} (free) $M$-CW complex is a CW complex built by attaching projective (free) $M$-cells, that is,  $X=\varinjlim X_n$ where $X_{-1}=\emptyset$ and $X_n$ is obtained from $X_{n-1}$  as a pushout
\begin{equation}\label{eq:pushout}
\begin{tikzcd}\coprod_{a\in A_n} Me_a\times S^{n-1}\ar{r}{\Psi_n}\ar[d, hook] & X_{n-1}\ar[d, hook]\\ \coprod_{a\in A_n} Me_a\times B^n\ar{r} & X_n\end{tikzcd}
\end{equation}
 where $\Psi_n$ is $M$-equivariant and continuous (and all $e_a=1$ in the free case).

If $X$ is an $M$-space, then $M\backslash X$ is the quotient of $X$ by the smallest equivalence relation identifying $x$ and $mx$ for all $x\in X$ and $m\in M$.  If $X$ is a projective $M$-CW complex, then $M\backslash X$ is a CW complex with $(M\backslash X)_n=M\backslash X_n$.  Note that  $M\backslash X_0=A_0$ and $M\backslash X_n$ is obtained from $M\backslash X_{n-1}$ by attaching an $n$-cell for each $a\in A_n$ via the composition $S^{n-1}\to e_a\times S^{n-1}\to X_{n-1}\to M\backslash X_{n-1}$.  In particular, $M\backslash X$ has $|A_n|$ $n$-cells.

From this it follows that the cellular chain complex for $X$ is given as a $\mathbb ZM$-module by $C_n(X)=\bigoplus_{a\in A_n} \mathbb ZMe_a$ and  $C_n(M\backslash X) = C_n(X)_M$ since $(\mathbb ZMe_a)_M\cong \mathbb Z[M\backslash Me_a]$ is a free abelian group on the coset of $e_a$.
Note that $\mathbb ZMe$ is a projective $\mathbb ZM$-module if $e\in M$ is an idempotent as $\mathbb ZM=\mathbb ZMe\oplus \mathbb ZM(1-e)$.

An \emph{equivariant classifying space} for $M$ is a contractible projective $M$-CW complex $X$.  Equivariant classifying spaces are unique up to $M$-homotopy equivalence~\cite[Theorem~6.3]{TopFinite1} (which I do not define here).    Note that the augmented cellular chain complex $C_{\bullet}(X)\twoheadrightarrow \mathbb Z$ of $X$ is a projective resolution of $\mathbb Z$.  Hence $H_n(M)\cong H_n(M\backslash X)$ and $H^n(M)=H^n(M\backslash X)$ by the previous paragraph.  A space $Y=M\backslash X$ with $X$ an equivariant classifying space for $M$ is called a \emph{classifying space} for $M$.  Any two classifying spaces for $M$ are homotopy equivalent~\cite[Corollary~6.7]{TopFinite1}.  If $G$ is a group, then a classifying space for $G$ is the same thing as an Eilenberg-Mac Lane space for $G$, and an equivariant classifying space is a universal cover of an Eilenberg-Mac Lane space.   Unlike the case of groups, a space homotopy equivalent to a classifying space for a monoid $M$ need not be a classifying space.  There is a standard classifying space $BM$ for $M$ obtained by taking the nerve of $M$, viewed as a one-object category.  The corresponding equivariant classifying space is $EM$, the nerve of the category of elements of the $M$-set $M$.  The cellular chain complex of $EM$ is the normalized bar resolution of $\mathbb Z$.

A monoid $M$ satisfies the \emph{topological finiteness property} (left) $\mathrm F_n$ if it has a classifying space with a finite $n$-skeleton (see~\cite[Section~6.1]{TopFinite1}).   One says $M$ satisfies the homological finiteness property (left) $\mathrm{FP}_n$ if $\mathbb Z$ has a projective resolution (as a left $\mathbb ZM$-module) that is finitely generated through degree $n$.  Note that $\mathrm F_n$ implies $\mathrm{FP}_n$, but not conversely.  Every finitely presented monoid is of type $\mathrm F_2$~\cite[Theorem~6.14]{TopFinite1} and, for monoids of type $\mathrm F_2$, the properties $\mathrm{FP}_n$ and $\mathrm{F}_n$ are equivalent~\cite[Theorem~6.15]{TopFinite1}.

If $X$ is a simply connected projective $M$-CW complex, then one can always embed $X$ in an equivariant classifying space $Y$ such that $Y_2=X_2$; see~\cite[Lemma~6.4]{TopFinite1}. If $M$ is given by a presentation $\langle X\mid R\rangle$, then the \emph{Cayley complex}  of this presentation is a simply connected free $M$-CW complex.  To describe it,   first note that the Cayley graph of $M$ with respect to $X$ is a free $M$-CW complex with a single free $M$-cell $M$ of vertices and, for each generator $x\in X$, it has a free $M$-cell $M\times B^1$ with $m\times B^1$ attached as the edge from $m$ to $mx$ labelled by $x$.
The Cayley complex $K_{X,R}$ is obtained from the Cayley graph by attaching a free $M$-cell $M\times B^2$ for each relation $u_i=v_i$ of $R$ with the boundary path of $m\times B^2$ the loop at $m$ labelled by  $u_iv_i\inv$.  The proof of~\cite[Theorem~6.14]{TopFinite1} shows that $K_{X,R}$ is simply connected.

Notice that $M\backslash K_{X,R}$ has a single vertex, $|X|$ loops at that vertex and $|R|$ $2$-cells attached by the loops labeled $u_iv_i\inv$ with $u_i=v_i$ a relation in $R$.  This is nothing more than the presentation complex for the \emph{group completion} $G(M)$ of $M$, that is, the group defined by the same presentation $\langle X\mid R\rangle$ viewed as a group presentation (this group doesn't depend on the presentation of $M$: it's the universal group receiving a homomorphism from $M$).  I'll now prove topologically Pride's unpublished theorem.  If $A$ is a finitely generated abelian group, then $\rk A =\dim_{\mathbb Q}\mathbb Q\otimes_{\mathbb Z} A$, i.e., the rank of the free abelian part of $A$.  Write $d(A)$ for the minimum number of generators of $A$.

\begin{Thm}[Pride]
Let $\langle X\mid R\rangle$ be a finite presentation of a monoid $M$.  Then $|R|-|X|\geq d(H_2(M))-\rk \ab{G(M)}$ where $G(M)$ is the group presented by this presentation.  The same result holds for semigroups.
\end{Thm}
\begin{proof}
Notice that viewing a semigroup presentation of $S$ as a monoid presentation, yields a monoid presentation of $S^1$.  Also note that the group completions $G(S)$ and $G(S^1)$ coincide.  Thus it suffices to prove the result for monoid presentations.  Let $K_{X,R}$ be the Cayley complex of $M$ with respect to the presentation $\langle X\mid R\rangle$.  By~\cite[Lemma~6.4]{TopFinite1}, there is an equivariant classifying space $\til Y$ for $M$ with $\til Y_2=K_{X,R}$, as $K_{X,R}$ is a simply connected free $M$-CW complex.  Let $Y=M\backslash \til Y$, and note that $Y_2$ is the presentation complex for $G(M)$ with respect to the group presentation $\langle X\mid R\rangle$, as discussed above. Then $H_1(M) = H_1(Y)=H_1(Y_2)=\ab{G(M)}$ by the Hurewicz theorem, since $G(M)=\pi_1(Y_2)$.  Also, $H_2(Y_2)$ is precisely the group $Z_2(Y)$ of $2$-cycles of $Y$, which moreover is finitely generated free abelian since $C_2(Y)$ is free abelian of rank $|R|$.  Since $H_2(M)=H_2(Y)=Z_2(Y)/B_2(Y)=H_2(Y_2)/B_2(Y)$, we deduce that $d(H_2(M))\leq \rk H_2(Y_2)$.  Then we perform the Euler characteristic computation
\begin{equation}\label{eq:eulerchar1}
\chi(Y_2) = 1-|X|+|R|
\end{equation}
 (as $Y_2$ is the presentation complex of $G(M)$).  But also
\begin{equation}\label{eq:eulerchar2}
\begin{split}
\chi(Y_2) &= 1-\rk H_1(Y_2)+\rk H_2(Y_2) \\ &= 1-\rk \ab{G(M)}+\rk H_2(Y_2)\geq 1-\rk \ab{G(M)}+d(H_2(M)).
\end{split}
\end{equation}
Comparing \eqref{eq:eulerchar1} and \eqref{eq:eulerchar2} establishes the theorem.
\end{proof}

The proof shows that if $K$ is a classifying space for $M$, then $\pi_1(K)\cong G(M)$ (as was already observed in~\cite{TopFinite1}).
Of course, if $M$ is finite, or more generally is finitely generated and has no nontrivial homomorphisms to $\mathbb Z$, then $\ab{G(M)}$ is a torsion group, and so $\rk \ab{G(M)}=0$.  This yields the following corollary.

\begin{Cor}[Pride]
Let $\langle X\mid R\rangle$ be a finite presentation of a finite monoid (or a monoid with no nonntrivial homomorphisms to $\mathbb Z$).  Then  the inequality $|R|-|X|\geq d(H_2(M))$ holds.  An analogous result holds for semigroups.
\end{Cor}

For later purposes I shall need  two more constructions from~\cite{TopFinite1} that allow us to build projective $M$-CW complexes.  If $X$ is a projective $M$-CW complex, then a projective $M$-CW subcomplex of $X$ is a subcomplex $Y$ that is a union of projective $M$-cells.  Retaining the notation of \eqref{eq:pushout}, this means that, for each $n\geq 0$, there is a partition $A_n=A_n'\coprod A_n''$ so $Y_n$ is obtained from $Y_{n-1}$ by attaching $\coprod_{a\in A_n'} Me_a\times B^n$ by the restriction of $\Psi_n$.  If $X$ and $Z$ are projective $M$-CW complexes, $Y$ is a projective $M$-CW subcomplex of $X$ and $f\colon Y\to Z$ is an $M$-equivariant cellular map, then it is proved in~\cite[Lemma~2.1]{TopFinite1} that the adjunction space $X\coprod_f Z$ obtained by gluing $Y$ to $Z$ via $f$ is a projective $M$-CW complex.

If $M,N$ are two monoids, $X$ is a left $N$-space and $A$ is an $M$-$N$-biset, then $A\otimes_N X$ is the quotient of $A\times X$ by the smallest equivalence relation $\sim$ such that $(an,x)\sim (a,nx)$ for all $a\in A$, $n\in N$ and $x\in X$.  The class of $(a,x)$ is denoted $a\otimes x$.  Note that $A\otimes_N X$ is a left $M$-space with action $m(a\otimes x) = ma\otimes x$.  If $A$ is a free right $N$-set with basis $B$, then $A\otimes_N X\cong B\times X$ via $bn\otimes x\mapsto (b,nx)$.  Here $B$ is a basis if the natural map $B\times N\to A$ sending $(b,n)$ to $bn$ is an isomorphism of right $N$-sets (where $(b,n)n'=(b,nn')$).  It is shown in~\cite[Corollary~3.2]{TopFinite1} that if $A$ is an $M$-$N$-biset that is projective as a left $M$-set (i.e., $A\cong \coprod_{b\in B}Me_b$ with $e_b\in M$ idempotent) and $Y$ is a projective $N$-CW complex, then $A\otimes_N Y$ is a projective $M$-CW complex with $(A\otimes_N Y)_n=A\otimes_N Y_n$. 

\section{Homology of completely simple semigroups}
It is now high time to prove Theorem~A.   First I will prove the topological version of Theorem~A, which will also allow me to connect the topological finiteness properties of a completely simple semigroup with those of its maximal subgroup.

\begin{Thm}\label{t:main.top.version}
Let $S=\mathcal M(G,A,B,C)$ be a Rees matrix semigroup over the group $G$ and assume that $C$ is normalized with respect to $a_0\in A$ and $b_0\in B$.  Put $\ov A=A\setminus \{a_0\}$, $\ov B=B\setminus \{b_0\}$ and set $M=S^1$.  Let $Y$ be a classifying space for $G$.  Fix a base point $y_0\in Y$ and choose a combinatorial loop $\gamma_{ba}\colon S^1\to Y$ at $y_0$ representing $C_{ba}$ in $\pi_1(Y,y_0)=G$, for each $(a,b)\in \ov A\times \ov B$ (with $\gamma_{ba}$ the constant path if $C_{ba}=1$).  Let  $f=\coprod_{(a,b)\in \ov A\times \ov B}\gamma_{ba}\colon \coprod_{(a,b)\in \ov A\times \ov B}S^1\to Y$.   Then there is a classifying space $X$ for $M$ that is homotopy equivalent to the adjunction space $(\coprod_{(a,b)\in \ov A\times \ov B} B^2)\coprod_f Y$ via a homotopy equivalence contracting a subcomplex isomorphic to $\bigvee_{a\in\ov A} B^2$.
\end{Thm}
\begin{proof}
Let $\til Y$ be the universal covering space of $Y$.  Note that $\til Y$ is a free contractible $G$-CW complex. Fix a preimage $\til y_0\in \til Y$ of the base point $y_0\in Y$.   For each $(a,b)\in \ov A\times \ov B$, let $\til \gamma_{ba}$ be the lift of $\gamma_{ba}$ beginning at $\til y_0$, and hence ending at $C_{ba}\til y_0$.  If $C_{ba}=1$, then $\til\gamma_{ba}$ is the constant path at $\til y_0$. Let us also put $\til \gamma_{b_0a}$ equal to the constant path at $\til y_0$ for $a\in \ov A$.

 If $b\in B$, let $L_b = A\times G\times \{b\}$.  Note that $e_b = (a_0,1,b)$ is an idempotent and $L_b=Me_b$.  Also, $G$ acts freely on the right of $L_b$ via the rule $(a,g,b)g' = (a,gg',b)$, and this action commutes with the left action of $M$.  Thus $L_b$ is an $M$-$G$-biset that is projective as a left $M$-set and free as a right $G$-set with basis $\{(a,1,b)\mid a\in A\}$ in bijection with $A$.  It follows that $L_{b_0}\otimes_G \til Y$ is a projective $M$-CW complex (by~\cite[Corollary~3.2]{TopFinite1}) that is homeomorphic to $A\times \til Y$ (with $A$ discrete)  via $(a,g,b_0)\otimes y\mapsto (a,gy)$.  It will be convenient to identify $\til Y$ with the copy $e_{b_0}\otimes \til Y$ (i.e., $a_0\times \til Y$) of $\til Y$ via $y\mapsto e_{b_0}\otimes y$.

I now construct a second projective $M$-CW complex $Z$ that I'll glue to $L_{b_0}\otimes_G \til Y$ along a projective $M$-CW subcomplex. Note that $S=\coprod_{b\in B} L_b$ is a projective left $M$-set.  Take $Z_0=L_{b_0}$.  To build $Z_1$ attach, for each $a\in \ov A$ and $m\in M$, an edge $m\epsilon_a\colon me_{b_0}\to m(a,1,b_0)$ and an additional edge $s\eta_a\colon se_{b_0}\to s(a,1,b_0)$ for each $s\in S$. This amounts to attaching $\coprod_{a\in \ov A} M\times B^1$ and $\coprod_{a\in \ov A}S\times B^1$. The notation is chosen so that $m'\cdot (m\epsilon_a)=(m'm)\epsilon_a$ and $m'\cdot (s\eta_a) = (m's)\eta_a$.   We then attach $\coprod_{a\in \ov A} S\times B^2$ by attaching, for each $s\in S$ and $a\in \ov A$, a $2$-cell $sc_a$ with boundary path $s\eta_a(s\epsilon_a)\inv$.  Again, $m\cdot sc_a=(ms)c_a$.   Note that the subcomplex $Z'$ consisting of all the vertices and all the edges of the form $s\eta_a$ is a projective $M$-CW subcomplex.  We can define an $M$-equivariant cellular map $F\colon Z'\to L_{b_0}\otimes_G \til Y$ by $F(s)=s\til y_0$ (i.e., $s\otimes \til y_0$) on vertices $s\in L_{b_0}$ and $F((a',g,b)\eta_a) = (a',g,b)\til \gamma_{ba}$.  This is well defined because $(a',g,b)\eta_a$ is an edge from $(a',g,b_0)$ to $(a',gC_{ba},b_0)$ and $(a',g,b)\til \gamma_{ba}$ is  path from $(a',g,b)\otimes y_0=(a',g,b)\til y_0$ to $(a',g,b)\otimes C_{ba}\til y_0 = (a',gC_{ba},b_0)\otimes \til y_0$.  It is clearly $M$-equivariant and cellular.

Let $\til X=Z\coprod_F (L_{b_0}\otimes_G \til Y)$; it is a projective $M$-CW complex by~\cite[Lemma~2.1]{TopFinite1}.  I claim it is contractible.  Observe that $\til X$ is obtained from $L_{b_0}\otimes_G \til Y$ by attaching edges $m\epsilon_a\colon m\til y_0\to m(a,1,b_0)\til y_0$ for $a\in \ov A$ and $m\in M$ and attaching a $2$-cell $(a',g,b)c_a$ with boundary path $(a',g,b)(\til\gamma_{ba}\epsilon_a\inv)$ for each  for $a\in \ov A$, $(a',g,b)\in S$.  Note that $(a',g,b)\epsilon_a$ is an edge $(a',g,b)\til y_0\to (a',g,b)((a,1,b_0)\otimes \til y_0)=(a',g,b_0)\otimes C_{ba}\til y_0$, so that this makes sense.

From now on it will be convenient to identify $L_{b_0}\otimes_G \til Y$ with $A\times \til Y$ (with $a_0\times \til Y$ identified with $\til Y$).  The edge $\epsilon_a$ goes from $(a_0,\til y_0)$ to $(a,\til y_0)$. Thus the subcomplex $W$ consisting of $A\times \til Y\cup \{\epsilon_a\mid a\in \ov A\}$ is obtained from attaching to each vertex of a star graph with $|\ov A|$-edges a copy of $\til Y$ at $\til y_0$.  The star graph is a tree, and so can be contracted without changing the homotopy type of $W$.  Thus $W$ is homotopy equivalent to a wedge of $|A|$ copies of $\til Y$ and hence is also contractible as $\til Y$ is contractible.  So $W$ can be contracted  without changing the homotopy type of $\til X$.   Doing so leaves us with a wedge of disks consisting of the single vertex to which $W$ is contracted, the edges $s\epsilon_a$ with $s\in S$ and $a\in \ov A$ and the $2$-cells $sc_a$ with $s\in S$ and $a\in \ov A$, which have respective boundary paths $s\epsilon_a\inv$ after contracting $W$.  But a wedge of disks is contractible.  Thus $\til X$ is an equivariant classifying space for $M$.

Next I'll describe the classifying space $X=M\backslash \til X$.  Note that if $D$ is an $M$-set, then $M\backslash D\cong \{\ast\}\otimes_M D$, where $\{\ast\}$ the trivial right $M$-set.  Thus associativity of tensor products yields $M\backslash (L_{b_0}\otimes_G \til Y)\cong (M\backslash L_{b_0})\otimes_G \til Y\cong \{\ast\}\otimes_G \til Y \cong G\backslash \til Y\cong Y$ as $|M\backslash L_{b}|=1$ for any $b\in B$.  Note that $M\backslash S=\coprod_{b\in B}M\backslash L_b$ is in bijection with $B$.  Since tensor products commute with colimits, $M\backslash \til X$ is obtained from $Y$ by gluing in a loop $\ov \epsilon_a$ at $y_0$ for each $a\in \ov A$ and $2$-cells $d_{ba}$, for $a\in \ov A$, $b\in B$ where $d_{b_0a}$ has boundary path $\ov \epsilon_a\inv$ and $d_{ba}$ has boundary path $\gamma_{ba}\ov \epsilon_a\inv$ if $b\in \ov B$.  The subcomplex consisting of the base point, the loops $\ov \epsilon_a$ and the $2$-cells $d_{b_0a}$ is a wedge of $|\ov A|$ disks.  Contracting this wedge yields a CW complex  homotopy equivalent to $X$ that is precisely  the adjunction space $(\coprod_{(a,b)\in \ov A\times \ov B} B^2)\coprod_f Y$, completing the proof.
\end{proof}

Observe that Theorem~\ref{t:main.top.version} implies the well-known fact that $G(S)\cong \pi_1(X)\cong G/\langle C_{ba}\mid (a,b)\in \ov A\times \ov B\rangle^G$.

Note that if $S$ is a finite rectangular band (meaning $G$ is trivial and $A,B$ are finite), then Theorem~\ref{t:main.top.version} shows that $M=S^1$ has a classifying space that is homotopy equivalent to a wedge of $(|A|-1)(|B|-1)$ $2$-spheres.  In particular, if $S$ is a $2\times 2$ rectangular band, then $M$ has a classifying space homotopy equivalent to a $2$-sphere, as was first noted in~\cite{Fiedor}.  If $\langle X\mid R\rangle$ is  group presentation and $S=\mathcal M(F_X,2,R\cup \{\ast\},C)$ with $C_{r1}=1$, $C_{r2}=r$ for $r\in R$ and $C_{\ast 1}=1=C_{\ast 2}$, then $M$ has a classifying space homotopy equivalent to the presentation $2$-complex associated to $\langle X\mid R\rangle$.

Let us now turn to the proof of Theorem~A.

\begin{proof}[Proof of Theorem~A]
I'll keep the notation of Theorem~\ref{t:main.top.version}.  Let $X'$ be the adjunction space $(\coprod_{(a,b)\in \ov A\times \ov B} B^2)\coprod_f Y$ and consider the long exact sequence in cellular homology associated to the pair $(X',Y)$.  Of course, $H_0(M)\cong H_0(X')\cong\mathbb Z$.  Note that \[C_n(X',Y) \cong \begin{cases} \bigoplus_{(a,b)\in \ov A\times \ov B}\mathbb Z, & \text{if}\ n=2\\ 0, & \text{else,}\end{cases}\] since $Y$ contains everything except for the $2$-cells attached via $f$.  Moreover, the image under the boundary map of the $2$-cell $(a,b)\times B^2$ is the $1$-cycle of $Y$ given by the signed sum of the oriented edges  of $\gamma_{ba}$.  Thus the connecting map $\partial\colon H_2(X',Y)=C_2(X',Y)\to H_1(Y)\cong \ab G$ sends the $2$-cell $(a,b)\times B^2$ to $C_{ba}[G,G]$ under the Hurewicz isomorphism.  Therefore, the long exact sequence in relative homology yields $H_n(M)\cong H_n(X')\cong H_n(Y)\cong H_n(G)$ for $n\geq 3$ and an exact sequence \[0\to H_2(G)\to H_2(M)\to  \mathbb Z[\ov A\times \ov B]\xrightarrow{\psi} \ab G\to H_1(M)\to 0.\]  This splits into two exact sequences
$0\to H_2(G)\to H_2(M)\to \ker \psi\to 0$ and $0\to \image \psi\to \ab G\to H_1(M)\to 0$.  Thus $H_1(M) \cong \coker \psi$.  Since $\ker\psi$ must be free abelian (being a subgroup of $\mathbb Z[\ov A\times \ov B]$), the first sequence splits yielding $H_2(M)\cong H_2(G)\oplus \ker \psi$. 
\end{proof}

Theorem~A has an analogue for cohomology (which is a slightly more detailed version of the results of~\cite{Nicocohom1}) with a similar proof that  I omit.

\begin{Thm}
Let $S=\mathcal M(G,A,B,C)$ be a completely simple semigroup over a group $G$.  Assume that $C$ is normalized with respect to $a_0\in A$ and $b_0\in B$ and put $\ov A=A\setminus\{a_0\}$, $\ov B=B\setminus \{b_0\}$.  Let $\Psi\colon \Hom_{\mathbb Z}(\ab G,\mathbb Z)\to \mathbb Z^{\ov A\times \ov B}$ be the homomorphism with $\Psi(f)(a,b) =f(C_{ba}[G,G])$.  Then
\[H^n(S) = \begin{cases}\mathbb Z, & \text{if}\ n=0\\ \ker \Psi=\Hom_{\mathbb Z}(\ab{G(M)},\mathbb Z), & \text{if}\ n=1\\  H_n(G), & \text{if}\ n\geq 3\end{cases}\] and there is an exact sequence $0\to \coker \Psi\to H^2(M)\to H^2(G)\to 0$.   In particular, if $\left\langle C_{ba}[G,G]\mid (a,b)\in \ov A\times \ov B\right\rangle$ is finite, then $\Psi$ is the zero map, and hence $H^1(M)\cong \Hom_{\mathbb Z}(\ab G,\mathbb Z)$, and there is an exact sequence $0\to \mathbb Z^{\ov A\times \ov B}\to H^2(M)\to H^2(G)\to 0$.
\end{Thm}

In particular, if $G$ is trivial, then $H^2(M)\cong \mathbb Z^{\ov A\times \ov B}$, as was observed in~\cite{Nicocohom1}.  I suspect that it is not generally the case that $H^2(M)\cong \coker \Psi\oplus H^2(G)$, but I don't have an example.

\subsection*{An algebraic proof of Theorem~A}
I'll now sketch an algebraic proof of most of Theorem~A.  Chasing down all the isomorphisms and connecting maps should allow one to produce exactly Theorem~A, but it hardly seems worth the effort.  If $V$ is a left $\mathbb ZM$-module, then $H_n(M,V)=\Tor^{\mathbb ZM}_n(\mathbb Z,V)$.  Let $S=\mathcal M(G,A,B,C)$ be a Rees matrix semigroup over the group $G$ and assume that $C$ is normalized with respect to $a_0\in A$ and $b_0\in B$.  Put $\ov A=A\setminus \{a_0\}$, $\ov B=B\setminus \{b_0\}$ and set $M=S^1$.

Keeping the notation from the previous section, note that $L_{b_0}/G$ is in bijection with $A$.  Thus we can define an action of $M$ on $A$ so that $m(a,1,b_0)G = (ma,1,b_0)G$.  Then we have an exact sequence
\begin{equation}\label{eq:exact}
0\to K\to \mathbb ZA\xrightarrow{\epsilon} \mathbb Z\to 0
\end{equation}
 where $\epsilon(a)=1$ for all $a\in A$.  Note that $K$ has basis as an abelian group all elements of the form $a-a_0$ with $a\neq a_0$.   Moreover, $s(a-a_0) =0$ for all $s\in S$.  Thus $K\cong \bigoplus_{a\in \ov A}\mathbb ZM/\mathbb ZS$, whence $H_n(M,K)\cong \bigoplus_{a\in \ov A}H_n(M,\mathbb ZM/\mathbb ZS)$.   Let us now compute $H_n(M,\mathbb ZM/\mathbb ZS)$ using the projective resolution $0\to \mathbb ZS\to \mathbb ZM\to \mathbb ZM/\mathbb ZS\to 0$.   Then $H_\bullet(M,\mathbb ZM/\mathbb ZS)$ is the homology of the chain complex $\mathbb ZB\xrightarrow{\eta} \mathbb Z$ where $\eta(b)=1$ all $b\in B$ (as $\mathbb ZS_M\cong \mathbb ZB$ since $M\backslash S$ is in bijection with $B$).
  Therefore, we have
\[H_n(M,\mathbb ZM/\mathbb ZS)\cong \begin{cases} \bigoplus_{b\in \ov B}\mathbb Z, & \text{if}\ n=1\\ 0, & \text{else.}\end{cases}\]

 The last ingredient is to compute $H_n(M,\mathbb ZA)$.  I claim $H_n(M,\mathbb ZA)\cong H_n(G)$.  To see this, note that if $P_\bullet\twoheadrightarrow \mathbb Z$ is a projective resolution of $\mathbb Z$ over $\mathbb ZG$, then since $\mathbb ZL_{b_0}$ is a $\mathbb ZM$-$\mathbb ZG$-bimodule that is projective as a left $\mathbb ZM$-module and free, hence flat, as a right $\mathbb ZG$-module, it follows that $\mathbb ZL_{b_0}\otimes_{\mathbb ZG} P_\bullet$ is a projective resolution of $\mathbb ZL_{b_0}\otimes_{\mathbb ZG}\mathbb Z\cong \mathbb ZA$ over $\mathbb ZM$.  Thus
 \begin{align*}
 H_n(M,\mathbb ZA)&=\Tor^{\mathbb ZM}_n(\mathbb Z,\mathbb ZA)=H_n(\mathbb Z\otimes_{\mathbb ZM} (\mathbb ZL_{b_0}\otimes_{\mathbb ZG}P_\bullet))\\ &=H_n((\mathbb Z\otimes_{\mathbb ZM}\mathbb ZL_{b_0})\otimes_{\mathbb ZG} P_{\bullet})\cong H_n(\mathbb Z\otimes_{\mathbb ZG}P_{\bullet})=H_n(G)
 \end{align*}
  as $|M\backslash L_{b_0}|=1$.

   Applying the long exact sequence for homology (i.e., for $\Tor$), and using the above computations, yields a long exact sequence
\[\cdots \to H_{n+1}(M)\to \bigoplus_{a\in \ov A} H_n(M,\mathbb ZM/\mathbb ZS)\to H_n(G)\to H_n(M)\to \cdots\]  It follows that $H_n(M)\cong H_n(G)$ for $n\geq 3$ and there is an exact sequence \[0\to H_2(G)\to H_2(M)\to  \bigoplus_{(a,b)\in \ov A\times \ov B}\mathbb Z\xrightarrow{\partial} H_1(G)\to H_1(M)\to 0.\]  Note that $H_2(M)\cong H_2(G)\oplus \ker \partial$ as $\ker \partial$ is free abelian.  Also, $H_1(G)\cong \ab G$, and $H_1(M)\cong \coker \partial$.  Thus to recover Theorem~A, it suffices to check that $\partial$ can be identified with $\psi$. This involves chasing a number of isomorphisms and the construction of the connecting homomorphism in the long exact sequence for $\Tor$, and is left to the interested reader.

\section{Homological and topological finiteness properties}
Gray and Pride considered homological finiteness properties for completely simple semigroups in~\cite{GrayPride}.  In particular, they were able to resolve the situation entirely when $B$ is finite and to describe the homological finiteness property $\mathrm{FP}_1$ in general.  I'll now complete their work by describing precisely when a completely simple semigroup is of type $\mathrm{FP}_n$ using a mixture of algebraic and topological techniques.  I can describe the situation for the topological finiteness property $\mathrm{F}_n$ in a satisfactory manner when $B$ is finite, but I can only give a partial result when $B$ is infinite.

Recall that if $R$ is a ring, then a left $R$-module  is of type $\mathrm{FP}_n$ if it has a projective resolution that is finitely generated through degree $n$.  It is of type $\mathrm{FP}_{\infty}$ if it is of type $\mathrm{FP}_n$ for all $n\geq 0$ or, equivalently, if it has a projective resolution that is finitely generated in each degree.  In particular, $\mathrm{FP}_0$ and $\mathrm{FP}_1$ are the properties of being, respectively, finite generated and finite presented. Of course, a monoid $M$ is of type $\mathrm{FP}_n$ precisely when the trivial $\mathbb ZM$-module $\mathbb Z$ is of type $\mathrm{FP}_n$.
The following result is~\cite[Proposition~1.4]{Bieribook}.

\begin{Thm}\label{t:bieri}
Let $R$ be a ring. 
If $0\to U\to V\to W\to 0$ is an exact sequence of left $R$-modules, then:
\begin{enumerate}
  \item If $U$ is of type $\mathrm{FP}_{n-1}$ and $V$ is of type $\mathrm{FP}_n$, then $W$ is of type $\mathrm{FP}_n$.
  \item If $V$ is of type $\mathrm{FP}_{n-1}$ and $W$ is of type $\mathrm{FP}_n$, then $U$ is of type $\mathrm{FP}_{n-1}$.
  \item If $U$ and $W$ are of type $\mathrm{FP}_n$, then $V$ is of type $\mathrm{FP}_n$.
\end{enumerate}
\end{Thm}

The following is a reformulation of~\cite[Proposition~4.3]{Browncohomology}.

\begin{Prop}\label{p:partial.res}
Let $M$ be an $R$-module and $n\geq 0$.  Then $M$ is of type $\mathrm{FP}_n$ if and only if $M$ is finitely generated and, for each exact sequence $0\to K\to P_k\to P_{k-1}\to \cdots \to P_0\to M\to 0$ with $0\leq k<n$ and $P_0,\ldots, P_k$ finitely generated projective modules, $K$ is finitely generated.
\end{Prop}

 Let us retain the notation of the previous section.   In particular, $S=\mathcal M(G,A,B,C)$ is a Rees matrix semigroup over the group $G$, $C$ is normalized with respect to $a_0\in A$ and $b_0\in B$ and $M=S^1$.  Let $\ov A=A\setminus \{a_0\}$, $\ov B=B\setminus \{b_0\}$ and  $e=(a_0,1,b_0)$.
The following was observed in~\cite{GrayPride}, but I'll argue differently.

\begin{Prop}\label{p:Afinite}
If $M$ is of type $\mathrm{FP}_1$, then $A$ is finite.
\end{Prop}
\begin{proof}
 Theorem~\ref{t:bieri}(2) applied to the exact sequence \eqref{eq:exact} implies that $K\cong \bigoplus_{a\in \ov A}\mathbb ZM/\mathbb ZS$ is finitely generated, and hence $A$ is finite, as $\mathbb ZA$ is a cyclic $\mathrm ZM$-module.
\end{proof}

The following construction underlies my approach to homological finiteness properties.  Assume from now on that $A$ is finite and fix a subset $Y\subseteq G$.  Let us build a projective $M$-$1$-complex $\Gamma_S(Y)$ as follows.  Take $\Gamma_S(Y)_0=L_{b_0}$.  Attach, for each $a\in \ov A$,  a free $1$-cell $M\times B^1$ by adding edges $m\epsilon_a$ from $m(a_0,1,b_0)$ to $m(a,1,b_0)$, for $m\in M$.  Notice that $\epsilon_a$ is an edge from $(a_0,1,b_0)$ to $(a,1,b_0)$.  If $s=(a',g,b)$ with $a'\in A$, then $s\epsilon_a$ is an edge from $(a',g,b_0)$ to $(a',gC_{ba},b_0)$.  Also attach a projective $1$-cell $L_{b_0}\times B^1$, for each $y\in Y$, by adding an edge $s\eta_y$ from $s(a_0,1,b_0)$ to $s(a_0,y,b_0)$.  So if $s=(a',g,b_0)$, then $s\eta_y$ is an edge from $(a',g,b_0)$ to $(a',gy,b_0)$.

Thus  $\Gamma_S(Y)$ consists of $|A|$ disjoint copies of the Cayley graph of $G$ with respect to the set $Y\cup \{C_{ba}\mid (a,b)\in \ov A\times B\}$ (with vertex sets $a\times G\times b_0$ with $a\in A$), attached together by the edges $\epsilon_a$ from $(a_0,1,1)$ to $(a,1,1)$ for $a\in \ov A$.    Here the vertex $g$ of the Cayley graph corresponds to $(a,g,b_0)$ in the $a$-copy, the edge $g$ to $gy$ labelled by $y\in Y$ corresponds to $(a,g,b_0)\eta_y$ in the $a$-copy, and the edge $g$ to $gC_{ba'}$ corresponds to $(a,g,b)\epsilon_{a'}$ in the $a$-copy. Notice that the action of $M$ on $\Gamma_S(Y)$ restricts to an action of $G$ on the $a_0$-copy of its Cayley graph that is just the usual left action of $G$ on its Cayley graph.  The following result is~\cite[Theorem~7]{GrayPride} with a different proof.

\begin{Prop}[Gray-Pride]\label{p:gp7}
Let $S=\mathcal M(G,A,B,C)$ be a completely simple semigroup, normalized as above.  Then $M=S^1$ is of type $\mathrm{FP}_1=\mathrm{F}_1$ if and only if $A$ is finite and there is a finite subset $Y\subseteq G$ such that $G=\left\langle Y\cup \{C_{ba}\mid (a,b)\in \ov A\times B\}\right\rangle$.
\end{Prop}
\begin{proof}
Assume that $A$ is finite, as this is necessary for $\mathrm{FP}_1$ by Proposition~\ref{p:Afinite}.
First observe that if $Y$  is subset of $G$, then $\Gamma_S(Y)$ is connected if and only if $G=\left\langle Y\cup \{C_{ba}\mid (a,b)\in \ov A\times B\}\right\rangle$ since $\Gamma_S(Y)$ is connected if and only the Cayley graph of $G$ with respect to $Y\cup \{C_{ba}\mid (a,b)\in \ov A\times B\}$  is connected by the above discussion.   Suppose first that there is a finite subset $Y\subseteq G$ such that $G=\left\langle Y\cup \{C_{ba}\mid (a,b)\in \ov A\times B\}\right\rangle$.  Then, since $\Gamma_S(Y)$ is connected, its augmented cellular chain complex gives a finite projective presentation $\bigoplus_{y\in Y}\mathbb ZL_{b_0}\oplus \bigoplus_{a\in \ov A}\mathbb ZM\to \mathbb ZL_{b_0}\to\mathbb Z$, and so $\mathbb Z$ is of type $\mathrm{FP}_1$. (Alternatively, $M$ is of type $\mathrm{F}_1$ by~\cite[Proposition~6.10]{TopFinite1}.)

For the converse, consider the connected projective $M$-$1$-complex $\Gamma_S(G)$.  Note that $C_0(\Gamma_S(G))= \mathbb ZL_{b_0}$ and let $d\colon C_1(\Gamma_S(G))\to \mathbb ZL_{b_0}$ be the boundary map, with image $B_0$.  Then $0\to B_0\to \mathbb ZL_{b_0}\to \mathbb Z\to 0$ is exact, $\mathbb Z$ is of type $\mathrm{FP}_1$ and  $\mathbb ZL_{b_0}$ is cyclic, hence $B_0$ is finitely generated by Theorem~\ref{t:bieri}(2).  As a $\mathbb ZM$-module, $C_1(\Gamma_S(G))$ is generated by the edges $\epsilon_a$ with $a\in \ov A$ and the edges $e\eta_g$ with $g\in G$.  Since any generating set of a finitely generated module contains a finite generating subset, there is some finite subset $Y\subseteq G$ such that $B_0$ is generated as a $\mathbb ZM$-module by the images under $d$ of the edges $\epsilon_a$ with $a\in \ov A$ and $e\eta_y$ with $y\in Y$.  But this means precisely that the subcomplex $\Gamma_S(Y)$ of $\Gamma_S(G)$ is connected and hence $G=\left\langle Y\cup \{C_{ba}\mid (a,b)\in \ov A\times B\}\right\rangle$.  This completes the proof.
\end{proof}

To complete the discussion of the homological finiteness properties of completely simple semigroups, I need to introduce a few more tools.  Suppose that $H$ is a group and $\psi\colon F\to H$ is a surjective homomorphism with $F$ a free group.  Put $R=\ker \psi$.  Then $R$ is a free subgroup, and so $\ab R$ is a free abelian group.  The conjugation action of $F$ on $R$ induces a $\mathbb ZF$-module structure on $\ab R$, which is in fact a $\mathbb ZH$-module structure since $R$ acts trivially on $\ab R$.  The $\mathbb ZH$-module $\ab R$ is called the \emph{relation module} associated to $\psi$.  Suppose that $X$ is a free basis for $F$.  Then it is well known (and seems to go back to Mac Lane as far as I can tell) that $\ab R\cong H_1(\Gamma(H,X))$ where $\Gamma(H,X)$ is the Cayley graph of $H$ with respect to $X$.  Roughly speaking, $\Gamma(H,X)$ is the covering space of the wedge of $|X|$ circles corresponding to the subgroup $R$ and hence $H_1(\Gamma(H,X))\cong \ab{\pi_1(\Gamma(H,X))}\cong \ab R$ as abelian groups by the Hurewicz theorem.  It is easy to check this is a module isomorphism.  See~\cite[Page~43]{Browncohomology} for details.  An immediate consequence is the following.

\begin{Prop}\label{p:relation.module}
Let $\psi\colon F\to H$ be a surjective homomorphism with $F$ a finitely generated free group and $R=\ker \psi$.  Then, $H$ is of type $\mathrm{FP}_n$ with $n\geq 2$ if and only if the relation module $\ab R$ is of type $\mathrm{FP}_{n-2}$.
\end{Prop}
\begin{proof}
Let $X$ be a free generating set of $F$.
From $\ab R\cong  H_1(\Gamma(H,X))$ (or~\cite[Proposition~5.4]{Browncohomology}), there is an exact sequence $0\to \ab R\to \bigoplus_{x\in X}\mathbb ZH\xrightarrow{d} \mathbb ZH\to \mathbb Z\to 0$ of $\mathbb ZG$-modules.  Letting $B_0=\image d$, we have exact sequences $0\to \ab R\to \bigoplus_{x\in X}\mathbb ZH\to B_0\to 0$ and $0\to B_0\to \mathbb ZH\to \mathbb Z\to 0$.  Since $X$ is finite, it follows from Theorem~\ref{t:bieri} that $\mathbb Z$ is of type $\mathrm{FP}_n$ if and only if $B_0$  is of type $\mathrm{FP}_{n-1}$, if and only if $\ab R$ is of type $\mathrm{FP}_{n-2}$ for $n\geq 2$.  
\end{proof}

The following technical proposition will be of use.

\begin{Prop}\label{p:tensor.f0}
Let $W$ be a $\mathbb ZG$-module.  If $\mathbb ZL_{b_0}\otimes_{\mathbb ZG} W$ is finitely generated, then $W$ is finitely generated.
\end{Prop}
\begin{proof}
Let $V=\mathbb ZL_{b_0}\otimes_{\mathbb ZG} W$.  First note that $eV\cong e\mathbb ZM\otimes_{\mathbb ZM}(\mathbb ZMe\otimes_{\mathbb ZG} W)\cong (e\mathbb ZM\otimes_{\mathbb ZM}\mathbb ZMe)\otimes_{\mathbb ZG} W\cong \mathbb Z[eMe]\otimes_{\mathbb ZG} W\cong W$ by identifying $Me$ with $L_{b_0}$ and $eMe$ with $G$.   Let $\bigoplus_{J}\mathbb ZG\to W\to 0$ be exact.  Then, tensoring on the left by $\mathbb ZL_{b_0}$, yields $\bigoplus_{J}\mathbb ZL_{b_0}\to V\to 0$ is exact.  Since $V$ is finitely generated, there is a finite subset $J_0\subseteq J$ with $\bigoplus_{J_0}\mathbb ZL_{b_0}\to V\to 0$ exact.  Then $\bigoplus_{J_0}\mathbb ZG\to W\to 0$ is exact as $eV\cong W$, $\mathbb ZG\cong \mathbb Z[eMe]\cong e\mathbb ZL_{b_0}$ and $e(-)\cong \Hom_{\mathbb ZM}(\mathbb ZM,(-))$ is exact.  Thus $W$ is finitely generated.
\end{proof}

I now have the tools to show that homological finiteness properties are preserved and reflected by tensoring with $\mathbb ZL_{b_0}$.

\begin{Lemma}\label{l:crucial}
Let $W$ be a $\mathbb ZG$-module.  Then $W$ is of type $\mathrm{FP}_n$ if and only if $\mathbb ZL_{b_0}\otimes_{\mathbb ZG} W$ is of type $\mathrm{FP}_n$.
\end{Lemma}
\begin{proof}
Let $V=\mathbb ZL_{b_0}\otimes_{\mathbb ZG} W$.
First observe that since $\mathbb ZL_{b_0}=\mathbb ZMe$ is a finitely generated projective $\mathbb ZM$-module, if $P$ is a finitely generated projective $\mathbb ZG$-module, then $\mathbb ZL_{b_0}\otimes_{\mathbb ZG} P$ is a finitely generated projective $\mathbb ZM$-module.  Moreover, since $\mathbb ZL_{b_0}$ is free as a right $\mathbb ZG$-module it is flat.    In particular, if $W$ is of type $\mathrm{FP}_n$, then  there is a projective resolution $P_\bullet \twoheadrightarrow W$ with $P_i$ finitely generated for $0\leq i\leq n$.  Then $\mathbb ZL_{b_0}\otimes_{\mathbb ZG}P_\bullet\twoheadrightarrow  V$ is a projective resolution with $\mathbb ZL_{b_0}\otimes_{\mathbb ZG}P_i$ finitely generated for $0\leq i\leq n$.  Thus  $V$ is of type $\mathrm{FP}_n$.

Conversely, assume that $V$ is of type $\mathrm{FP}_n$.  I'll prove that $W$ is of type $\mathrm{FP}_n$ using Proposition~\ref{p:partial.res}.  First note that $W$ is finitely generated by Proposition~\ref{p:tensor.f0} since $V$ is of type $\mathrm{FP}_0$. Suppose that $0\to K\to P_k\to P_{k-1}\to \cdots\to P_0\to W\to 0$ is an exact sequence with $P_0,\ldots, P_k$ finitely generated projective modules and $0\leq k<n$.  Tensoring with the free right $\mathbb ZG$-module $\mathbb ZL_{b_0}$ yields an exact sequence $0\to \mathbb ZL_{b_0}\otimes_{\mathbb ZG}K\to \mathbb ZL_{b_0}\otimes_{\mathbb ZG}P_k\to\cdots\to \mathbb ZL_{b_0}\otimes_{\mathbb ZG}P_0\to V\to 0$ with $\mathbb ZL_{b_0}\otimes_{\mathbb ZG}P_i$ finitely generated projective for $0\leq i\leq k$ by the previous paragraph.  Since $V$ is of type $\mathrm{FP}_n$, it follows that $ \mathbb ZL_{b_0}\otimes_{\mathbb ZG}K$ is finitely generated by Proposition~\ref{p:partial.res}, and hence $K$ is finitely generated by Proposition~\ref{p:tensor.f0}.   Thus $W$ is of type $\mathrm{FP}_n$ by Proposition~\ref{p:partial.res}.
\end{proof}

The following theorem completes the investigation of homological finiteness conditions for completely simple semigroups begun in~\cite{GrayPride}.

\begin{Thm}\label{t:fpthm}
Let $S=\mathcal M(G,A,B,C)$ be a completely simple semigroup over a group $G$.  Assume that $C$ is normalized with respect to $a_0\in A$, $b_0\in B$, and put $\ov A=A\setminus\{a_0\}$.  Then $M=S^1$ is of type $\mathrm{FP}_n$ with $n\geq 1$, if and only  if $A$ is finite and there is a finite subset $Y\subseteq G$ such that if $F$ is the free group on $Y\bigcup \{C_{ba}\mid a\in \ov A\times B\}$, then the natural map $\gamma\colon F\to G$ is onto and, if $R=\ker \gamma$ and $n\geq 2$, then the relation module $\ab R$ is of type $\mathrm{FP}_{n-2}$ as a $\mathbb ZG$-module.  This latter condition is independent of the choice of $Y$.
\end{Thm}
\begin{proof}
Proposition~\ref{p:gp7} already handles the case $n=1$, so assume that $n\geq 2$, $A$ is finite and $Y\subseteq G$ is finite such that $\gamma$ is onto.  Consider the projective $M$-$1$-complex $\Gamma_S(Y)$ constructed above. Its augmented cellular chain complex $C_1(\Gamma_S(Y))\xrightarrow{d} C_0(\Gamma_S(Y))\to \mathbb Z$ is then a finite projective presentation of $\mathbb Z$.  Let $B_0=\image d$.  Then we have exact sequences $0\to H_1(\Gamma_S(Y))\to C_1(\Gamma_S(Y))\to B_0\to 0$ and $0\to B_0\to C_0(\Gamma_S(Y))\to \mathbb Z\to 0$.  Since $C_0(\Gamma_S(Y))$ and $C_1(\Gamma_S(Y))$ are finitely generated projectives, we have that $\mathbb Z$ is of type $\mathrm{FP}_n$ if and only if, $B_0$ is of type $\mathrm{FP}_{n-1}$, if and only if $H_1(\Gamma_S(Y))$ is of type $\mathrm{FP}_{n-2}$ for $n\geq 2$ by Theorem~\ref{t:bieri}.   I claim that $H_1(\Gamma_S(Y))\cong \mathbb ZL_{b_0}\otimes_{\mathbb ZG} \ab R$.  The result will then follow from Lemma~\ref{l:crucial}.

To see this,  note that $\Gamma_S(Y)$ has an $M$-invariant subcomplex $Z$  obtained by removing the edges $\epsilon_a$ with $a\in \ov A$.    Note that $Z$ is not a projective sub-$M$-CW complex, but it is isomorphic to a projective $M$-CW complex by modifying the construction of $\Gamma_S(Y)$ to glue in projective $M$-cells $S\times B^1$ of the form $s\epsilon_a$ instead of $M\times B^1$ for each $a\in \ov A$.  By our previous discussion, $Z$ is topologically $A\times \Gamma(G,X)$ where $X=Y\cup \{C_{ba}\mid (a,b)\in \ov A\times B\}$.  Notice that $\Gamma_S(Y)/Z$ is a star graph with vertex set $A$ and an edge $\epsilon_a$ from $a_0$ to $a$ for each $a\in \ov A$, hence contractible.  Using the natural isomorphism $H_n(\Gamma_S(Y),Z)\cong \til H_n(\Gamma_S(Y)/Z)$, we see from the long exact relative homology sequence (which is natural) that $H_1(\Gamma_S(Y))\cong H_1(Z)$ as $\mathbb ZM$-modules.  I claim that $H_1(Z)\cong \mathbb ZL_{b_0}\otimes_{\mathbb ZG}H_1(\Gamma(G,X))$.  Since it was already observed that $H_1(\Gamma(G,X))\cong \ab R$, this will complete the proof.

Let $Z_a$ be the $a$-copy of $\Gamma(G,X)$ in $Z$.  Then $Z=\coprod_{a\in A} Z_a$ is the decomposition into path components.  So $H_1(Z)=\bigoplus_{a\in A}H_1(Z_a)$ as an abelian group.
As mentioned earlier, the action of $M$ on $\Gamma_S(Y)$ (and hence $Z$) restricts to the usual action of $G$ on $a_0$-copy of $\Gamma(G,X)$ identified with $\Gamma(G,X)$.  There results an isomorphism of $\mathbb ZG$-modules $\p\colon H_1(\Gamma(G,X))\to H_1(Z_{a_0})\subseteq H_1(Z)$.  Note that $eH_1(Z)=H_1(Z_{a_0})$, and so by the hom-tensor adjunction (as $e(-)\cong \Hom_{\mathbb ZM}(\mathbb ZMe,(-))=\Hom_{\mathbb ZM}(\mathbb ZL_{b_0},(-)$) there is an induced homomorphism $\Phi\colon \mathbb ZL_{b_0}\otimes_{\mathbb ZG}H_1(\Gamma(G,X))\to H_1(Z)$ such that $\Phi(e\otimes c) = \p(c)$.  Note that since $L_{b_0}$ is a free right $G$-set with basis $\{(a,1,b_0)\mid a\in A\}$, it follows that $\mathbb ZL_{b_0}\otimes_{\mathbb ZG} H_1(\Gamma(G,X))\cong \bigoplus_{a\in A}(a,1,b_0)\otimes H_1(\Gamma(G,X))$ as abelian groups.   Note that $(a,1,b_0)(a_0,1,b_0)=(a,1,b_0)$ and $(a_0,1,b_0)(a,1,b_0)=(a_0,1,b_0)$ implies that left multiplication by $(a,1,b_0)$ takes $Z_{a_0}$ isomorphically to $Z_a$ with inverse left multiplication by $(a_0,1,b_0)$ and hence the same is true on homology.  Also left multiplication by $(a,1,b_0)$ takes $(a_0,1,b_0)\otimes H_1(\Gamma(G,X))$ isomorphically to $(a,1,b_0)\otimes H_1(\Gamma(G,X))$ with inverse left multiplication by $(a_0,1,b_0)$.  It now follows that since $\Phi((a_0,1,b_0)\otimes c)=\p(c)$ and $\p$ is an isomorphism that $\Phi$ is an isomorphism.  In fact, identifying  $\mathbb ZL_{b_0}\otimes_{\mathbb ZG}H_1(\Gamma(G,X))$ with $\bigoplus_{a\in A}H_1(\Gamma(G,X))$ as an abelian group, $\Phi$ is the map taking $a$-component to the isomorphic group $H_1(Z_a)$ via the isomorphism $c\mapsto (a,1,b_0)\p(c)$.
\end{proof}

As a corollary, I obtain~\cite[Theorem~4]{GrayPride}.

\begin{Cor}
Let $S=\mathcal M(G,A,B,C)$ be a completely simple semigroup over a group $G$.   Suppose that $B$ is finite. Then $M=S^1$ is of type $\mathrm{FP}_n$ with $n\geq 1$ if and only  if $A$ is finite and $G$ is of type $\mathrm{FP}_n$.
\end{Cor}
\begin{proof}
Recall that a group is of type $\mathrm{FP}_1$ if and only if it is finitely generated~\cite{Browncohomology}. We may assume that $A$ is finite by
Proposition~\ref{p:Afinite}. Since $\{C_{ba}\mid (a,b)\in \ov A\times B\}$ (retaining the previous notation) is a finite set when $A$ and $B$ are finite, $G$ is of type $\mathrm{FP}_1$ if and only if there is a finite subset $Y\subseteq G$  such that if $F$ is the free group on $Y\bigcup \{C_{ba}\mid a\in \ov A\times B\}$, then the natural map $\gamma\colon F\to G$ is onto. Moreover,  if $R=\ker \gamma$ and $n\geq 2$, then the relation module $\ab R$ is of type $\mathrm{FP}_{n-2}$ if and only if $G$ is of type $\mathrm{FP}_n$ by Proposition~\ref{p:relation.module}.  The result now follows directly from Theorem~\ref{t:fpthm}.
\end{proof}

In the case that $B$ is infinite, the condition that $\ab R$ be of type $\mathrm{FP}_{n-2}$ doesn't seem all that easy to get one's hands on. Sometimes there are easier methods to see that $M$ is not of type $\mathrm{FP}_n$.
 If $M$ is of type $\mathrm{FP}_n$, then $H_k(M)$ is finitely generated for $k\leq n$.  This follows because if $P$ is a finitely generated projective $\mathbb ZM$-module, then $P_M$ is a finitely generated free abelian group.

\begin{Prop}
Retaining the notation of Theorem~A, if $M$ is of type $\mathrm{FP}_2$ and $B$ is infinite, then  $\langle C_{ba}[G,G]\mid (a,b)\in \ov A\times \ov B\rangle$  is not finitely generated, and hence $G$ is not finitely generated.
\end{Prop}
\begin{proof}
If $\langle C_{ba}[G,G]\mid (a,b)\in \ov A\times \ov B\rangle$ is finitely generated, then $\ker\psi$ from Theorem~A must have infinite rank since $B$ is infinite, and so $H_2(M)$ is not finitely generated.
\end{proof}


Let $G=\bigoplus_{i\geq 1}\mathbb Zx_i$.  Put $A=\{a_0,a_1\}$ and $B=\{b_0,b_1,\ldots\}$.  Put $C_{ba_0}=1 = C_{b_0a}$, for $a\in A$ and $b\in B$, and $C_{b_ia_1}=x_i$ for $i\geq 1$.  Then $G$ is not finitely generated, and so not of type $\mathrm{FP_1}$.   But $M$ is of type $\mathrm{FP}_1$ by~\cite[Theorem~7]{GrayPride} or Proposition~\ref{p:gp7}.  However, $M$ is not of type $\mathrm{FP}_2$.  To see this, note that $\psi\colon \mathbb Z\ov B\to G$ given by $\psi(b_i) = C_{b_ia_1}$ is an isomorphism.  Thus $H_1(S) \cong 0$ and $H_n(S) \cong H_n(G)$ for $n\geq 2$ by Theorem~A.   But $\mathbb Z^d$ is retract of $G$ for each $d\geq 2$, and so $H_2(\mathbb Z^d)$ is a quotient of $H_2(G)$.  But $H_2(\mathbb Z^d)$ is free abelian of rank $\binom{d}{2}$.  Therefore, $H_2(M)\cong H_2(G)$ is not finitely generated, whence $M$ is not of type $\mathrm{FP}_2$. 

On the other hand, if $A$ is finite, $B$ is infinite and if the $C_{ba}$ with $a\in \ov A$ and $b\in \ov B$ freely generate $G$ (as in~\cite[Example~2]{GrayPride}), then   $\gamma\colon F\to G$ has kernel the normal subgroup generated by the $C_{b_0a}$ with $a\in \ov A$, and it is not hard to check that $\ab R$ is freely generated by these elements as a $\mathbb ZG$-module, and is hence of type $\mathrm{FP}_{\infty}$.  Thus $M$ is of type $\mathrm{FP}_{\infty}$ in this case by Theorem~\ref{t:fpthm}.

The topological finiteness condition $\mathrm F_n$ is straightforward to characterize when $B$ is finite using Theorem~\ref{t:main.top.version}.

\begin{Thm}
Let  $S=\mathcal M(G,A,B,C)$ be a Rees matrix semigroup over the group $G$ and assume that $C$ is normalized with respect to $a_0\in A$ and $b_0\in B$ and that $B$ is finite.  Put $M=S^1$.  Then the following are equivalent for $n\geq 1$.
\begin{enumerate}
\item $M$ has the topological finiteness property $\mathrm F_n$.
\item $A$ is finite and $G$ has the topological finiteness property $\mathrm F_n$.
\end{enumerate}
\end{Thm}
\begin{proof}
Suppose that $M$ has the topological finiteness property $\mathrm F_n$.  Then it has the homological finiteness property $\mathrm{FP}_1$, and so $A$ is finite by~\cite[Theorem~7]{GrayPride}.  Suppose that $\til X$ is an equivariant classifying space for $M$ with $M\backslash \til X_n$ finite.  Let $e=(a_0,1,b_0)$.  Then $eM=a_0\times G\times B$ is a free left $G$-set with $|B|$ orbits under the action $g'(a_0,g,b) = (a_0,g'g,b)$.  It follows that if $f$ is any idempotent, then $eMf\subseteq eM$ is $G$-invariant and hence a free $G$-set with fewer than $|B|$ orbits (actually, it's a transitive $G$-set if $f\neq 1$).  Thus $e\til X\cong eM\otimes_M \til X$ is a free $G$-CW complex and  $G\backslash e\til X_n$ is a finite CW complex by~\cite[Corollary~3.2]{TopFinite1}.  But $e\til X$ is a retract of $\til X$ via $x\mapsto ex$, and a retract of a contractible space is contractible.  Thus $G\backslash \til X$ is a classifying space for $G$, and so $G$ has the topological finiteness property $\mathrm F_n$.

Suppose conversely, that $A$ is finite and $G$ has the topological finiteness property $\mathrm F_n$.  Let $Y$ be a classifying space for $G$ with finite $n$-skeleton.  Then Theorem~\ref{t:main.top.version} implies that $M$ has a classifying space with a finite $n$-skel\-e\-ton and hence property $\mathrm F_n$, completing the proof.
\end{proof}

I'll now try to handle the case where $B$ is infinite.  I'll use the following special case of the main result of~\cite{Hannebauer}.  Suppose that $H=G_1\ast G_2$ where $G_1$ is generated by $X_1$ and $G_2$ is generated by $X_2$.  If $\ab R$ is the relation module of $H$ with respect to the generators $X_1\cup X_2$, and $\ab {(R_i)}$ is the relation module of $G_i$ with respect to the generators $X_i$, for $i=1,2$, then $\ab R\cong (\mathbb ZH\otimes_{\mathbb ZG_1}\ab {(R_1)})\oplus (\mathbb ZH\otimes_{\mathbb ZG_2}\ab {(R_2)})$.
I suspect the sufficient condition in the next theorem is, in fact, necessary.

\begin{Thm}\label{t:top.finiteness}
Let $S=\mathcal M(G,A,B,C)$ be a completely simple semigroup over a group $G$.  Assume that $C$ is normalized with respect to $a_0\in A$ and $b_0\in B$, and set $\ov A=A\setminus\{a_0\}$. Suppose that $G$ has a presentation $\langle Y\bigcup \{C_{ba}\mid a\in \ov A\times B\}\mid r_1,\ldots, r_k\rangle$ with $Y\subseteq G$ finite.  Then $M=S^1$ is of type $\mathrm{F}_2$.  Moreover, if the subgroup $H$ of $G$ with presentation $\langle X\mid r_1,\ldots, r_k\rangle$ is of type $\mathrm F_n$ for $n\geq 2$, where $X$ consists of $Y$ and the finitely many $C_{ba}$ that appear in some relator $r_1,\ldots, r_k$, then $M$ is of type $\mathrm F_n$.
\end{Thm}
\begin{proof}
We shall add finitely many projective $M$-cells of dimension $2$ to $\Gamma_S(Y)$ to obtain a simply connected projective $M$-CW complex.  This will imply that $M$ is of type $\mathrm F_2$ by~\cite[Proposition~6.10]{TopFinite1}.  Recall that $\Gamma_S(Y)$ looks like a disjoint union of $|A|$ copies of the Cayley graph of $G$ with respect to the generating set $Y\bigcup \{C_{ba}\mid a\in \ov A\times B\}$ attached by the edges $\epsilon_a$ from $(a_0,1,b_0)$ to $(a,1,b_0)$ with $a\in \ov A$.  Moreover, left multiplication by $(a,1,b_0)$ provides the isomorphism from the $a_0$-copy to the $a$-copy (preserving the labelling of both the edges and the vertices).  It follows that if we glue in $k$ projective  $M$-cells $L_{b_0}\times B^2$ by attaching $2$-cells $sc_i$ with boundary path the circuit at $s\in L_{b_0}$ labeled by $r_i$ (under our identification of edges of $\Gamma_S(Y)$ other than the $\epsilon_a$ with edges of the Cayley graph of $G$), then we obtain a projective $M$-CW complex which topologically is $|A|$ disjoint copies of the universal cover of the presentation complex of  $\langle Y\bigcup \{C_{ba}\mid a\in \ov A\times B\}\mid r_1,\ldots, r_k\rangle$ attached by the edges $\epsilon_a$ from $(a_0,1,b_0)$ to $(a,1,b_0)$ with $a\in \ov A$, which is simply connected.  It follows that $M$ is of type $\mathrm F_2$.  If $n\geq 2$, then $M$ is of type $\mathrm F_n$ if and only if it is of type $\mathrm{FP}_n$ by~\cite[Theorem~6.15]{TopFinite1}.

Next observe that $G=H\ast L$ where $L$ is a free group on those entries $C_{ba}$ with $(a,b)\in \ov A\times B$ that do not appear in any of $r_1,\ldots r_k$.  Notice that  the relation module for $L$ with respect to its free generators is $0$.   By~\cite{Hannebauer}, if $V$ is the relation module for the presentation $\langle X\mid r_1,\ldots, r_k\rangle$ of $H$, then for $\gamma\colon F\to G$ as in Theorem~\ref{t:fpthm} with $R=\ker \gamma$, one has $\ab R\cong \mathbb ZG\otimes_{\mathbb ZH}V$.    Thus if $V$ is of type $\mathrm{FP}_{n-2}$, then so is $\ab R$ (as $\mathbb ZG$ is a free right $\mathbb ZH$-module and hence tensoring with $\mathbb ZG$ preserves projectivity, finite generation and resolutions).  Thus $M$ is of type $\mathrm{FP}_n$ by Theorem~\ref{t:fpthm}.

\end{proof}

\def\malce{\mathbin{\hbox{$\bigcirc$\rlap{\kern-7.75pt\raise0,50pt\hbox{${\tt
  m}$}}}}}\def\cprime{$'$} \def\cprime{$'$} \def\cprime{$'$} \def\cprime{$'$}
  \def\cprime{$'$} \def\cprime{$'$} \def\cprime{$'$} \def\cprime{$'$}
  \def\cprime{$'$} \def\cprime{$'$}


\end{document}